\documentclass{amsart}
\usepackage{amssymb}

\newtheorem{theorem}{Theorem}[section]
\newtheorem{proposition}[theorem]{Proposition}
\newtheorem{lemma}[theorem]{Lemma}

\newtheorem{cor}[theorem]{Corollary}
\theoremstyle{definition}
\newtheorem{definition}[theorem]{Definition}

\numberwithin{equation}{section}

\begin{document}
\title[Certain character sums and hypergeometric series]
{Certain character sums and hypergeometric series}


\author{Rupam Barman}
\address{Department of Mathematics, Indian Institute of Technology Guwahati, North Guwahati, Guwahati-781039, Assam, INDIA}
\curraddr{}
\email{rupam@iitg.ac.in}
\author{Neelam Saikia}
\address{Department of Mathematics, Indian Institute of Technology Guwahati, North Guwahati, Guwahati-781039, Assam, INDIA}
\curraddr{}
\email{neelam16@iitg.ernet.in}
\thanks{}


\date{13th February 2018}
\thanks{Acknowledement: This work is partially
supported by a start up grant of the first author awarded by Indian Institute of Technology Guwahati. The second author acknowledges
the financial support of Department of Science and Technology, Government of India for supporting a part of this work under INSPIRE Faculty Fellowship.}
\begin{abstract} We prove two transformations for the $p$-adic hypergeometric series which can be described as $p$-adic analogues of a Kummer's linear transformation and a
transformation of Clausen. We first evaluate two character sums, and then relate them to the $p$-adic hypergeometric series to deduce the transformations.
We also find another transformation for the $p$-adic hypergeometric series from which many special values of the $p$-adic hypergeometric series as well as
finite field hypergeometric functions are obtained.
\end{abstract}
\maketitle
\section{Introduction and statement of results}

For a complex number $a$, the rising factorial or the Pochhammer symbol is defined as $(a)_0=1$ and $(a)_k=a(a+1)\cdots (a+k-1), ~k\geq 1$.
For a non-negative integer $r$, and $a_i, b_i\in\mathbb{C}$ with $b_i\notin\{\ldots, -3,-2,-1\}$,
the classical hypergeometric series ${_{r+1}}F_{r}$ is defined by
\begin{align}
{_{r+1}}F_{r}\left(\begin{array}{cccc}
                   a_1, & a_2, & \ldots, & a_{r+1} \\
                    & b_1, & \ldots, & b_r
                 \end{array}| \lambda
\right):=\sum_{k=0}^{\infty}\frac{(a_1)_k\cdots (a_{r+1})_k}{(b_1)_k\cdots(b_r)_k}\cdot\frac{\lambda^k}{k!},\notag
\end{align}
which converges for $|\lambda|<1$.
Throughout the paper $p$ denotes an odd prime and $\mathbb{F}_q$ denotes the finite field with $q$ elements, where $q=p^r, r\geq 1$.
Greene \cite{greene} introduced the notion of hypergeometric functions over finite fields analogous to the classical hypergeometric series. Finite field
hypergeometric series were developed mainly to simplify character sum evaluations. Let $\widehat{\mathbb{F}_q^{\times}}$ be the group of all multiplicative
characters on $\mathbb{F}_q^{\times}$. We extend the domain of each $\chi\in \widehat{\mathbb{F}_q^{\times}}$ to $\mathbb{F}_q$ by setting $\chi(0)=0$
including the trivial character $\varepsilon$. For multiplicative characters $A$ and $B$ on $\mathbb{F}_q$,
the binomial coefficient ${A \choose B}$ is defined by
\begin{align}\label{eq-0}
{A \choose B}:=\frac{B(-1)}{q}J(A,\overline{B})=\frac{B(-1)}{q}\sum_{x \in \mathbb{F}_q}A(x)\overline{B}(1-x),
\end{align}
where $J(A, B)$ denotes the usual Jacobi sum and $\overline{B}$ is the character inverse of $B$.
Let $n$ be a positive integer.
For characters $A_0, A_1,\ldots, A_n$ and $B_1, B_2,\ldots, B_n$ on $\mathbb{F}_q$,
Greene defined the ${_{n+1}}F_n$ finite field hypergeometric functions over $\mathbb{F}_q$ by
\begin{align}
{_{n+1}}F_n\left(\begin{array}{cccc}
                A_0, & A_1, & \ldots, & A_n\\
                 & B_1, & \ldots, & B_n
              \end{array}\mid x \right)_q
              =\frac{q}{q-1}\sum_{\chi\in \widehat{\mathbb{F}_q^\times}}{A_0\chi \choose \chi}{A_1\chi \choose B_1\chi}
              \cdots {A_n\chi \choose B_n\chi}\chi(x).\notag
\end{align}
\par 
Some of the biggest motivations for studying finite field hypergeometric functions have
been their connections with Fourier coefficients and eigenvalues of modular forms and
with counting points on certain kinds of algebraic varieties. Their  links  to  Fourier  coefficients and eigenvalues  of
modular forms are established by many authors, for example, see \cite{ahlgren, evans-mod, frechette, fuselier, Fuselier-McCarthy, lennon, mccarthy4, mortenson}.
Very recently, McCarthy and Papanikolas \cite{mc-papanikolas} linked the finite field hypergeometric functions to Siegel modular forms. 
It is well-known that finite field hypergeometric functions can be used to count points on varieties over finite fields. For example, see \cite{BK, BK1, fuselier, koike, 
lennon2, ono, salerno, vega}.
\par
Since the multiplicative characters on $\mathbb{F}_q$ form a cyclic group of order $q-1$, a condition like $q\equiv 1 \pmod{\ell}$
must be satisfied where $\ell$ is the least common multiple of the orders of the characters appeared in the hypergeometric function.
Consequently, many results involving these functions are restricted to primes in certain congruence classes. To overcome these restrictions, McCarthy \cite{mccarthy1, mccarthy2}
defined a function
${_{n}}G_{n}[\cdots]_q$ in terms of quotients of the $p$-adic gamma function
which can best be described as an analogue of hypergeometric series in the $p$-adic setting (defined in Section 2).
\par Many transformations exist for  finite  field hypergeometric functions which are analogues of certain classical results \cite{greene, mccarthy3}.
Results involving  finite  field hypergeometric functions can readily be converted to expressions involving ${_{n}}G_{n}[\cdots]$.
However these new expressions in ${_{n}}G_{n}[\cdots]$ will be valid for the same set of primes for which the original expressions 
involving finite field hypergeometric functions existed.
 It is a non-trivial exercise to then extend these results to almost all primes.
There are very few identities and transformations for the $p$-adic hypergeometric series ${_{n}}G_{n}[\cdots]_q$ which exist for all but finitely many primes
(see for example \cite{BS2, BS1, BSM}).
Recently, Fuselier and McCarthy \cite{Fuselier-McCarthy} proved certain transformations for ${_{n}}G_{n}[\cdots]_q$, and used them to establish a supercongruence
conjecture of Rodriguez-Villegas between a truncated ${_4}F_3$ hypergeometric series and the Fourier coefficients of a certain weight four modular form.
\par Let $\chi_4$ be a character of order $4$. Then a finite field analogue of ${_2F}_{1}\left(\begin{array}{cc}
           \frac{1}{4}, & \frac{3}{4} \\
           &1
         \end{array}|x\right)$ is the function ${_2F}_{1}\left(\begin{array}{cc}
           \chi_4, & \chi_4^3 \\
           &\varepsilon
         \end{array}|x\right)$.
Using the relation between finite field hypergeometric functions and ${_n}G_n$-functions as given in Proposition \ref{prop-301} in Section 3,
the function ${_2G}_{2}\left[\begin{array}{cc}
           \frac{1}{4}, & \frac{3}{4} \\
           0, & 0
         \end{array}|\dfrac{1}{x}\right]_q$ can be described as a $p$-adic analogue of the classical hypergeometric series ${_2F}_{1}\left(\begin{array}{cc}
           \frac{1}{4}, & \frac{3}{4} \\
           &1
         \end{array}|x\right)$.
In this article, we prove the following transformation for the $p$-adic hypergeometric series which can be described as a $p$-adic analogue of
the Kummer's linear transformation \cite[p. 4, Eq. (1)]{bailey}.
Let $\varphi$ be the quadratic character on $\mathbb{F}_q$.
\begin{theorem}\label{MT-1}
Let $p$ be an odd prime and $x\in \mathbb{F}_q$. Then, for $x\neq 0, 1$, we have
\begin{align}
{_2G}_{2}\left[\begin{array}{cc}
           \frac{1}{4}, & \frac{3}{4} \\
           0, & 0
         \end{array}|\frac{1}{x}\right]_q=\varphi(-2){_2G}_{2}\left[\begin{array}{cc}
           \frac{1}{4}, & \frac{3}{4} \\
           0, & 0
         \end{array}|\frac{1}{1-x}\right]_q.\notag
\end{align}
\end{theorem}
We note that the finite field analogue of Kummer's linear transformation was discussed by Greene \cite[p. 109, Eq. (7.7)]{greene2} when $q\equiv 1\pmod{4}$.
\par We have $\varphi(-2)=-1$ if and only if $p\equiv 5, 7\pmod{8}$. Hence, using Theorem \ref{MT-1} for $x=\frac{1}{2}$, we obtain the following special value of the ${_n}G_n$-function.
\begin{cor} Let $p$ be a prime such that $p\equiv 5, 7\pmod{8}$. Then we have
\begin{align}\label{spv-1}
{_2G}_{2}\left[\begin{array}{cc}
           \frac{1}{4}, & \frac{3}{4} \\
           0, & 0
         \end{array}|2\right]_p=0.
         \end{align}
\end{cor}
If we convert the ${_n}G_n$-function given in \eqref{spv-1} using Proposition \ref{prop-301} in Section 3,
then we have ${_2}F_1\left(
         \begin{array}{cc}
           \chi_4, & \chi_4^3 \\
           ~ & \varepsilon \\
         \end{array}|\frac{1}{2}
       \right)_p=0$ for $p\equiv 5\pmod{8}$ which also follows from \cite[Eq. (4.15)]{greene}. The value of ${_2G}_{2}\left[\begin{array}{cc}
           \frac{1}{4}, & \frac{3}{4} \\
           0, & 0
         \end{array}|2\right]_p$ can be deduced from \cite[Eq. (4.15)]{greene} when $p\equiv 1 \pmod{8}$. It would be interesting to know the value of
         ${_2G}_{2}\left[\begin{array}{cc}
           \frac{1}{4}, & \frac{3}{4} \\
           0, & 0
         \end{array}|2\right]_p$ when $p\equiv 3 \pmod{8}$.

\par The following transformation for classical hypergeometric series is a special case
of Clausen's famous classical identity \cite[p. 86, Eq. (4)]{bailey}.
\begin{align}\label{clausen}
{_3}F_2\left(\begin{array}{ccc}
           \frac{1}{2}, & \frac{1}{2}, &\frac{1}{2} \\
           &1,&1
         \end{array}|x\right)=(1-x)^{-1/2}~{_2}F_1\left(\begin{array}{cc}
           \frac{1}{4}, & \frac{3}{4} \\
           &1
         \end{array}|\frac{x}{x-1}\right)^2.
\end{align}
A finite field analogue of \eqref{clausen} was studied by Greene \cite[p. 94, Prop. 6.14]{greene2}. In \cite{EG}, Evans and Greene gave a finite field analogue of the
Clausen's classical identity.
We prove the following transformation for the ${_n}G_n$-function which can be described as a $p$-adic analogue of \eqref{clausen}.
Let $\delta$ be the function defined on $\mathbb{F}_q$ by $$\delta(x)=\left\{
                                             \begin{array}{ll}
                                               1, & \hbox{if $x=0$;} \\
                                               0, & \hbox{if $x\neq0$.}
                                             \end{array}
                                           \right.
$$
\begin{theorem}\label{MT-4}
Let $p$ be an odd prime and $x\in\mathbb{F}_p$. Then, for $x\neq 0, 1$, we have
\begin{align}
{_3G}_{3}\left[\begin{array}{ccc}
\frac{1}{2}, & \frac{1}{2}, & \frac{1}{2} \\
0, & 0, & 0
\end{array}|\frac{1}{x}
\right]_p &=\varphi(1-x)\cdot
{_2G}_{2}\left[\begin{array}{cc}
           \frac{1}{4}, & \frac{3}{4} \\
           0, & 0
         \end{array}|\frac{x-1}{x}\right]_p^2-p\cdot\varphi(1-x).\notag
\end{align}
\end{theorem}
We also prove the following transformation using Theorem \ref{MT-1} and \cite[Thm. 4.16]{greene}.
\begin{theorem}\label{MT-3}
Let $p$ be an odd prime and $x\in\mathbb{F}_q$. Then, for $x\neq0, \pm 1$, we have
\begin{align}\label{eqn-MT-3}
{_2G}_{2}\left[\begin{array}{cc}
           \frac{1}{4}, & \frac{3}{4} \\
           0, & 0
         \end{array}|\frac{(1+x)^2}{(1-x)^2}\right]_q&=\varphi(-2)\varphi(1+x){_2G}_{2}\left[\begin{array}{cc}
           \frac{1}{2}, & \frac{1}{2} \\
           0, & 0
         \end{array}|x^{-1}\right]_q.
\end{align}
\end{theorem}
The following transformation is a finite field analogue of \eqref{eqn-MT-3}.
\begin{theorem}\label{MT-2}
Let $p$ be an odd prime and $q=p^r$ for some $r\geq1$ such that $q\equiv1\pmod{4}$. Then, for $x\neq0, \pm 1$, we have
\begin{align}
{_2}F_1\left(
         \begin{array}{cc}
           \chi_4, & \chi_4^3 \\
           ~ & \varepsilon \\
         \end{array}|\frac{(1-x)^2}{(1+x)^2}
       \right)_q=\varphi(-2)\varphi(1+x){_2}F_1\left(
         \begin{array}{cc}
           \varphi, & \varphi \\
           ~ & \varepsilon \\
         \end{array}|x
       \right)_q.\notag
\end{align}
\end{theorem}
Using Theorem \ref{MT-3} and Theorem \ref{MT-2}, one can deduce many special values of the $p$-adic hypergeometric series as well as the finite field hypergeometric functions.
For example, we have the following special values of a ${_n}G_n$-function and its finite field analogue.
\begin{theorem}\label{MT-5} For any odd prime $p$, we have
        \begin{align*}
{_2}G_2\left[\begin{array}{cc}
           \frac{1}{4}, & \frac{3}{4} \\
           0, & 0
         \end{array}|9\right]_p=\left\{
                                             \begin{array}{ll}
                                               0, & \hbox{if $p\equiv 3\pmod{4}$;} \\
                                               -2x\varphi(6)(-1)^{\frac{x+y+1}{2}}, & \hbox{if $p\equiv 1\pmod{4}$, $x^2+y^2=p$, and $x$ odd.}
                                             \end{array}
                                           \right.
\end{align*}
For $p\equiv1\pmod{4}$, we have
       \begin{align*}
       {_2}F_1\left(
         \begin{array}{cc}
           \chi_4, & \chi_4^3 \\
           ~ & \varepsilon \\
         \end{array}|\frac{1}{9}\right)_p=\frac{2x\varphi(6)(-1)^{\frac{x+y+1}{2}}}{p},
       \end{align*}
       where $x^2+y^2=p$ and $x$ is odd.
\end{theorem}
We also find special values of the following ${_n}G_n$-function.
\begin{theorem}\label{MT-7}
For $q\equiv 1\pmod{8}$ we have
       \begin{align}\label{eq-8.5}
       {_2}G_2\left[
         \begin{array}{cc}
           \frac{1}{4}, & \frac{3}{4} \\
           0, & 0 \\
         \end{array}|\left(\frac{6\sqrt{2}\pm3}{-2\sqrt{2}\pm3}\right)^2
       \right]_q=-q\varphi(6\pm 12\sqrt{2})
       \left\{{\chi_4 \choose \varphi}+{\chi_4^3 \choose \varphi}\right\}.
       \end{align}
      For $q\equiv 11\pmod{12}$ we have
       \begin{align}\label{eq-8.6}
       {_2}G_2\left[
         \begin{array}{cc}
           \frac{1}{4}, & \frac{3}{4} \\
           0, & 0 \\
         \end{array}|\left(\frac{6\pm\sqrt{3}}{-2\pm\sqrt{3}}\right)^2
       \right]_q=0.
\end{align}
For $q\equiv 1\pmod{12}$ we have
\begin{align}\label{eq-8.7}
{_2}G_2\left[
         \begin{array}{cc}
           \frac{1}{4}, & \frac{3}{4} \\
           0, & 0 \\
         \end{array}|\left(\frac{6\pm\sqrt{3}}{-2\pm\sqrt{3}}\right)^2
       \right]_q=-q\varphi\left(\frac{8\pm5\sqrt{3}}{12\pm 6\sqrt{3}}\right)
       \left\{{\varphi \choose \chi_3}+{\varphi \choose \chi_3^2}\right\}.
       \end{align}
\end{theorem}
The following theorem is a finite field analogue of Theorem \ref{MT-7}.
\begin{theorem}\label{MT-6}
For $q\equiv 1\pmod{8}$ we have
       \begin{align}\label{eq-8.3}
       {_2}F_1\left(
         \begin{array}{cc}
           \chi_4, & \chi_4^3 \\
           ~ & \varepsilon \\
         \end{array}|\left(\frac{-2\sqrt{2}\pm3}{6\sqrt{2}\pm3}\right)^2
       \right)_q=\varphi(6\pm 12\sqrt{2})
       \left\{{\chi_4 \choose \varphi}+{\chi_4^3 \choose \varphi}\right\}.
       \end{align}
For $q\equiv 1\pmod{12}$ we have
       \begin{align}\label{eq-8.4}
       {_2}F_1\left(
         \begin{array}{cc}
           \chi_4, & \chi_4^3 \\
           ~ & \varepsilon \\
         \end{array}|\left(\frac{-2\pm\sqrt{3}}{6\pm\sqrt{3}}\right)^2
       \right)_q=\varphi\left(\frac{8\pm5\sqrt{3}}{12\pm6\sqrt{3}}\right)
       \left\{{\varphi \choose \chi_3}+{\varphi \choose \chi_3^2}\right\}.
       \end{align}
\end{theorem}
In section 3 we prove two character sum identities and then use them to prove Theorem \ref{MT-1}, Theorem \ref{MT-4}, and Theorem \ref{MT-3}. We also prove
Theorem \ref{MT-2} in section 3. In section 4 we prove  Theorem \ref{MT-5}, Theorem \ref{MT-7} and Theorem \ref{MT-6}.
\section{Notations and Preliminaries}
Let $\mathbb{Z}_p$ and $\mathbb{Q}_p$ denote the ring of $p$-adic integers and the field of $p$-adic numbers, respectively.
Let $\overline{\mathbb{Q}_p}$ be the algebraic closure of $\mathbb{Q}_p$ and $\mathbb{C}_p$ the completion of $\overline{\mathbb{Q}_p}$.
Let $\mathbb{Z}_q$ be the ring of integers in the unique unramified extension of $\mathbb{Q}_p$ with residue field $\mathbb{F}_q$.
We know that $\chi\in \widehat{\mathbb{F}_q^{\times}}$ takes values in $\mu_{q-1}$, where $\mu_{q-1}$ is the group of
$(q-1)$-th roots of unity in $\mathbb{C}^{\times}$. Since $\mathbb{Z}_q^{\times}$ contains all $(q-1)$-th roots of unity,
we can consider multiplicative characters on $\mathbb{F}_q^\times$
to be maps $\chi: \mathbb{F}_q^{\times} \rightarrow \mathbb{Z}_q^{\times}$.
Let $\omega: \mathbb{F}_q^\times \rightarrow \mathbb{Z}_q^{\times}$ be the Teichm\"{u}ller character.
For $a\in\mathbb{F}_q^\times$, the value $\omega(a)$ is just the $(q-1)$-th root of unity in $\mathbb{Z}_q$ such that $\omega(a)\equiv a \pmod{p}$.
\par We now introduce some properties of Gauss sums. For further details, see \cite{evans}. Let $\zeta_p$ be a fixed primitive $p$-th root of unity
in $\overline{\mathbb{Q}_p}$. The trace map $\text{tr}: \mathbb{F}_q \rightarrow \mathbb{F}_p$ is given by
\begin{align}
\text{tr}(\alpha)=\alpha + \alpha^p + \alpha^{p^2}+ \cdots + \alpha^{p^{r-1}}.\notag
\end{align}
For $\chi \in \widehat{\mathbb{F}_q^\times}$, the \emph{Gauss sum} is defined by
\begin{align}
g(\chi):=\sum\limits_{x\in \mathbb{F}_q}\chi(x)\zeta_p^{\text{tr}(x)}.\notag
\end{align}
Now, we will see some elementary properties of Gauss and Jacobi sums.
We let $T$ denote a fixed generator of $\widehat{\mathbb{F}_q^\times}$.
\begin{lemma}\emph{(\cite[Eq. 1.12]{greene}).}\label{lemma2_1}
If $k\in\mathbb{Z}$ and $T^k\neq\varepsilon$, then
$$g(T^k)g(T^{-k})=qT^k(-1).$$
\end{lemma}
Let $\delta$ denote the function on multiplicative characters defined by
$$\delta(A)=\left\{
              \begin{array}{ll}
                1, & \hbox{if $A$ is the trivial character;} \\
                0, & \hbox{otherwise.}
              \end{array}
            \right.
$$
\begin{lemma}\emph{(\cite[Eq. 1.14]{greene}).}\label{lemma2_2} For $A,B\in\widehat{\mathbb{F}_q^{\times}}$ we have
\begin{align}
J(A,B)=\frac{g(A)g(B)}{g(AB)}+(q-1)B(-1)\delta(AB).\notag
\end{align}
\end{lemma}
The following are character sum analogues of the binomial theorem \cite{greene}.
For any $A\in\widehat{\mathbb{F}_q^{\times}}$ and $x\in\mathbb{F}_q$ we have
\begin{align}\label{lemma2_3}
\overline{A}(1-x)=\delta(x)+\frac{q}{q-1}\sum_{\chi\in\widehat{\mathbb{F}_q^{\times}}}{A\chi \choose \chi}\chi(x),
\end{align}
\begin{align}\label{lemma2_4}
A(1+x)=\delta(x)+\frac{q}{q-1}\sum_{\chi\in\widehat{\mathbb{F}_q^{\times}}}{A\choose \chi}\chi(x).
\end{align}
We recall some properties of the binomial coefficients from \cite{greene}. We have
\begin{align}\label{eq-1}
{A\choose B}={A\choose A\overline{B}},
\end{align}
\begin{align}\label{eq-2}
{A\choose \varepsilon}={A\choose A}=\frac{-1}{q}+\frac{q-1}{q}\delta(A).
\end{align}
\begin{theorem}\emph{(\cite[Davenport-Hasse Relation]{evans}).}\label{thm2_2}
Let $m$ be a positive integer and let $q=p^r$ be a prime power such that $q\equiv 1 \pmod{m}$. For multiplicative characters
$\chi, \psi \in \widehat{\mathbb{F}_q^\times}$, we have
\begin{align}
\prod\limits_{\chi^m=\varepsilon}g(\chi \psi)=-g(\psi^m)\psi(m^{-m})\prod\limits_{\chi^m=\varepsilon}g(\chi).\notag
\end{align}
\end{theorem}
Now, we recall the $p$-adic gamma function. For further details, see \cite{kob}.
For a positive integer $n$,
the $p$-adic gamma function $\Gamma_p(n)$ is defined as
\begin{align}
\Gamma_p(n):=(-1)^n\prod\limits_{0<j<n,p\nmid j}j\notag
\end{align}
and one extends it to all $x\in\mathbb{Z}_p$ by setting $\Gamma_p(0):=1$ and
\begin{align}
\Gamma_p(x):=\lim_{x_n\rightarrow x}\Gamma_p(x_n)\notag
\end{align}
for $x\neq0$, where $x_n$ runs through any sequence of positive integers $p$-adically approaching $x$.
This limit exists, is independent of how $x_n$ approaches $x$,
and determines a continuous function on $\mathbb{Z}_p$ with values in $\mathbb{Z}_p^{\times}$.
For $x \in \mathbb{Q}$ we let $\lfloor x\rfloor$ denote the greatest integer less than
or equal to $x$ and $\langle x\rangle$ denote the fractional part of $x$, i.e., $x-\lfloor x\rfloor$, satisfying $0\leq\langle x\rangle<1$.
We now recall the McCarthy's $p$-adic hypergeometric series $_{n}G_{n}[\cdots]$
as follows.
\begin{definition}\cite[Definition 5.1]{mccarthy2} \label{defin1}
Let $p$ be an odd prime and $q=p^r$, $r\geq 1$. Let $t \in \mathbb{F}_q$.
For positive integer $n$ and $1\leq k\leq n$, let $a_k$, $b_k$ $\in \mathbb{Q}\cap \mathbb{Z}_p$.
Then the function $_{n}G_{n}[\cdots]$ is defined by
\begin{align}
&_nG_n\left[\begin{array}{cccc}
             a_1, & a_2, & \ldots, & a_n \\
             b_1, & b_2, & \ldots, & b_n
           \end{array}|t
 \right]_q:=\frac{-1}{q-1}\sum_{a=0}^{q-2}(-1)^{an}~~\overline{\omega}^a(t)\notag\\
&\times \prod\limits_{k=1}^n\prod\limits_{i=0}^{r-1}(-p)^{-\lfloor \langle a_kp^i \rangle-\frac{ap^i}{q-1} \rfloor -\lfloor\langle -b_kp^i \rangle +\frac{ap^i}{q-1}\rfloor}
 \frac{\Gamma_p(\langle (a_k-\frac{a}{q-1})p^i\rangle)}{\Gamma_p(\langle a_kp^i \rangle)}
 \frac{\Gamma_p(\langle (-b_k+\frac{a}{q-1})p^i \rangle)}{\Gamma_p(\langle -b_kp^i \rangle)}.\notag
\end{align}
\end{definition}
Let $\pi \in \mathbb{C}_p$ be the fixed root of $x^{p-1} + p=0$ which satisfies
$\pi \equiv \zeta_p-1 \pmod{(\zeta_p-1)^2}$. Then the Gross-Koblitz formula relates Gauss sums and the $p$-adic gamma function as follows.
\begin{theorem}\emph{(\cite[Gross-Koblitz]{gross}).}\label{thm2_3} For $a\in \mathbb{Z}$ and $q=p^r$,
\begin{align}
g(\overline{\omega}^a)=-\pi^{(p-1)\sum\limits_{i=0}^{r-1}\langle\frac{ap^i}{q-1} \rangle}\prod\limits_{i=0}^{r-1}\Gamma_p\left(\langle \frac{ap^i}{q-1} \rangle\right).\notag
\end{align}
\end{theorem}
The following lemma relates products of values of $p$-adic gamma function.
\begin{lemma}\emph{(\cite[Lemma 3.1]{BS1}).}\label{lemma3_1}
Let $p$ be a prime and $q=p^r$. For $0\leq a\leq q-2$ and $t\geq 1$ with $p\nmid t$, we have
\begin{align}
\omega(t^{-ta})\prod\limits_{i=0}^{r-1}\Gamma_p\left(\langle\frac{-tp^ia}{q-1}\rangle\right)
\prod\limits_{h=1}^{t-1}\Gamma_p\left(\langle \frac{hp^i}{t}\rangle\right)
=\prod\limits_{i=0}^{r-1}\prod\limits_{h=0}^{t-1}\Gamma_p\left(\langle\frac{p^i(1+h)}{t}-\frac{p^ia}{q-1}\rangle \right).\notag
\end{align}
\end{lemma}
We now prove the following lemma which will be used to prove our results.
\begin{lemma}\label{lemma3.2}
Let $p$ be an odd prime and $q=p^r$. Then for $0\leq a\leq q-2$ and $0\leq i\leq r-1$ we have
\begin{align}\label{eq-3.10}
-\left\lfloor\frac{-4ap^i}{q-1}\right\rfloor +\left\lfloor\frac{-2ap^i}{q-1}\right\rfloor
=-\left\lfloor\langle\frac{p^i}{4}\rangle-\frac{ap^i}{q-1}\right\rfloor-\left\lfloor\langle\frac{3p^i}{4}\rangle-\frac{ap^i}{q-1}\right\rfloor.
\end{align}
\end{lemma}
\begin{proof}
Let $\left\lfloor\frac{-4ap^i}{q-1}\right\rfloor=4k+s$, where $k, s \in \mathbb{Z}$ satisfying $0\leq s\leq 3$.
Then we have
\begin{align}\label{eq-10.1}
4k+s\leq\frac{-4ap^i}{q-1}< 4k+s+1.
\end{align}
If $p^i\equiv 1\pmod{4}$, then \eqref{eq-10.1} yields
\begin{align}
\left\lfloor\frac{-2ap^i}{q-1}\right\rfloor=\left\{
                                              \begin{array}{ll}
                                                2k, & \hbox{if $s=0,1$;} \\
                                                2k+1, & \hbox{if $s=2,3$,}
                                              \end{array}
                                            \right.\notag
\end{align}
\begin{align}
\left\lfloor\langle\frac{p^i}{4}\rangle-\frac{ap^i}{q-1}\right\rfloor=\left\{
                                                                         \begin{array}{ll}
                                                                           k, & \hbox{if $s=0,1,2$;} \\
                                                                           k+1, & \hbox{if $s=3$,}
                                                                         \end{array}
                                                                       \right.\notag
\end{align}
and
\begin{align}
\left\lfloor\langle\frac{3p^i}{4}\rangle-\frac{ap^i}{q-1}\right\rfloor=\left\{
                                                                         \begin{array}{ll}
                                                                           k, & \hbox{if $s=0$;} \\
                                                                           k+1, & \hbox{if $s=1,2,3$.}
                                                                         \end{array}
                                                                       \right.\notag
\end{align}
Putting the above values for different values of $s$ we readily obtain \eqref{eq-3.10}. The proof of \eqref{eq-3.10} is similar when $p^i\equiv 3\pmod{4}$.
\end{proof}
\section{Proofs of the main results}
We first prove two propositions which enable us to express certain character sums in terms of the $p$-adic hypergeometric series.
\begin{proposition}\label{prop-1}
Let $p$ be an odd prime and $x\in\mathbb{F}_q^{\times}$. Then we have
\begin{align}
\sum_{y\in\mathbb{F}_q}\varphi(y)\varphi(1-2y+xy^2)& = \varphi(2x)+\frac{q^2\varphi(-2)}{q-1}\sum_{\chi\in\widehat{\mathbb{F}_q^{\times}}}
{\varphi\chi^2\choose \chi}{\varphi\chi\choose \chi}\chi\left(\frac{x}{4}\right)\notag\\
& = -\varphi(-2){_2G}_{2}\left[\begin{array}{cc}
           \frac{1}{4}, & \frac{3}{4} \\
           0, & 0
         \end{array}|\frac{1}{x}\right]_q.\notag
\end{align}
\end{proposition}
\begin{proof}
Applying \eqref{eq-1} and then \eqref{eq-0} we have
\begin{align}
\sum_{\chi\in\widehat{\mathbb{F}_q^{\times}}}
{\varphi\chi^2\choose \chi}{\varphi\chi\choose \chi}\chi\left(\frac{x}{4}\right)&=\sum_{\chi\in\widehat{\mathbb{F}_q^{\times}}}{\varphi\chi\choose \chi}\chi\left(\frac{x}{4}\right){\varphi\chi^2\choose \varphi\chi}\notag\\
&=\frac{\varphi(-1)}{q}\sum_{\chi\in\widehat{\mathbb{F}_q^{\times}}}{\varphi\chi\choose \chi}\chi\left(\frac{-x}{4}\right)J(\varphi\chi^2,\varphi\overline{\chi})\notag\\
&=\frac{\varphi(-1)}{q}\sum_{\substack{\chi\in\widehat{\mathbb{F}_q^{\times}}\\ y\in\mathbb{F}_q}}{\varphi\chi\choose \chi}\chi\left(\frac{-x}{4}\right)\varphi\chi^2(y)\varphi\overline{\chi}(1-y)\notag\\
&=\frac{\varphi(-1)}{q}\sum_{\substack{\chi\in\widehat{\mathbb{F}_q^{\times}}\\ y\in\mathbb{F}_q, y\neq 1}}
\varphi(y)\varphi(1-y){\varphi\chi\choose \chi}\chi\left(-\frac{xy^2}{4(1-y)}\right).\notag
\end{align}
Now, \eqref{lemma2_3} yields
\begin{align}
&\sum_{\chi\in\widehat{\mathbb{F}_q^{\times}}}
{\varphi\chi^2\choose \chi}{\varphi\chi\choose \chi}\chi\left(\frac{x}{4}\right)\notag\\
&=\frac{\varphi(-1)(q-1)}{q^2}\sum_{y\in\mathbb{F}_q, y\neq 1}
\varphi(y)\varphi(1-y)\left(\varphi\left(1+\frac{xy^2}{4(1-y)}\right)-\delta\left(\frac{xy^2}{4(1-y)}\right)\right)
\notag\\
&=\frac{(q-1)\varphi(-1)}{q^2}\sum_{y\in\mathbb{F}_q, y\neq 1}\varphi(y)\varphi(1-y)\varphi\left(1+\frac{xy^2}{4(1-y)}\right).\notag
\end{align}
Since $p$ is an odd prime, taking the transformation $y\mapsto 2y$ we obtain
\begin{align}
&\sum_{\chi\in\widehat{\mathbb{F}_q^{\times}}}
{\varphi\chi^2\choose \chi}{\varphi\chi\choose \chi}\chi\left(\frac{x}{4}\right)\notag\\
&=\frac{(q-1)\varphi(-2)}{q^2}\sum_{\substack{y\in\mathbb{F}_q\\ y\neq \frac{1}{2}}}\varphi(y)\varphi(1-2y)\varphi\left(1+\frac{xy^2}{1-2y}\right)\notag\\
&=\frac{(q-1)\varphi(-2)}{q^2}\sum_{\substack{y\in\mathbb{F}_q\\ y\neq \frac{1}{2}}}\varphi(y)\varphi(1-2y+xy^2)\notag\\
&=\frac{(q-1)\varphi(-2)}{q^2}\sum_{y\in\mathbb{F}_q}\varphi(y)\varphi(1-2y+xy^2)-\frac{\varphi(-x)(q-1)}{q^2},\notag
\end{align}
from which we readily obtain the first identity of the proposition.
\par To complete the proof of the proposition, we relate the above character sums to the $p$-adic hypergeometric series. From \eqref{eq-0}, Lemma \ref{lemma2_2},
and then using the facts that $\delta(\chi)=0$ for $\chi\neq \varepsilon, \delta(\varepsilon)=1$ and $g(\varepsilon)=-1$, we deduce that
\begin{align}
A & := \sum_{\chi\in\widehat{\mathbb{F}_q^{\times}}}
{\varphi\chi^2\choose\chi}{\varphi\chi\choose\chi}\chi\left(\frac{x}{4}\right)
=\frac{1}{q^2}\sum_{\chi\in\widehat{\mathbb{F}_q^{\times}}}J(\varphi\chi^2,\overline{\chi})J(\varphi\chi, \overline{\chi})\chi\left(\frac{x}{4}\right)\notag\\
&=\frac{1}{q^2}\sum_{\chi\in\widehat{\mathbb{F}_q^{\times}}}\frac{g(\varphi\chi^2)g^2(\overline{\chi})}{g(\varphi)}\chi\left(\frac{x}{4}\right)
+\frac{q-1}{q^2}\sum_{\chi\in\widehat{\mathbb{F}_q^{\times}}}\frac{g(\varphi\chi)g(\overline{\chi})}
{g(\varphi)}\chi\left(-\frac{x}{4}\right)\delta(\varphi\chi)\notag\\
&=\frac{1}{q^2}\sum_{\chi\in\widehat{\mathbb{F}_q^{\times}}}
\frac{g(\varphi\chi^2)g^2(\overline{\chi})}{g(\varphi)}\chi\left(\frac{x}{4}\right)-\frac{q-1}{q^2}\varphi(-x).\notag
\end{align}
Now, taking $\chi=\omega^a$ we have
\begin{align}
A&=\frac{1}{q^2}\sum_{a=0}^{q-2}\frac{g(\varphi\omega^{2a})g^2(\overline{\omega}^a)}{g(\varphi)}
\omega^a\left(\frac{x}{4}\right)-\frac{q-1}{q^2}\varphi(-x).\notag
\end{align}
Using Davenport-Hasse relation for $m=2$ and $\psi=\omega^{2a}$ we obtain
\begin{align}
g(\varphi\omega^{2a})=\frac{g(\omega^{4a})\overline{\omega}^{2a}(4)g(\varphi)}{g(\omega^{2a})}.\notag
\end{align}
Thus,
\begin{align}
A&=\frac{1}{q^2}\sum_{a=0}^{q-2}\omega^a(x)\overline{\omega}^{3a}(4)
\frac{g(\omega^{4a})g^2(\overline{\omega}^a)}{g(\omega^{2a})}-\frac{q-1}{q^2}\varphi(-x).\notag
\end{align}
Applying Gross-Koblitz formula we deduce that
\begin{align}
A&=\frac{1}{q^2}\sum_{a=0}^{q-2}\omega^a(x)\overline{\omega}^{3a}(4)\pi^{(p-1)\alpha}\prod_{i=0}^{r-1}
\frac{\Gamma_p(\langle\frac{-4ap^i}{q-1}\rangle)\Gamma_p^2(\langle\frac{ap^i}{q-1}\rangle)}{\Gamma_p(\langle\frac{-2ap^i}{q-1}\rangle)}
-\frac{q-1}{q^2}\varphi(-x),\notag
\end{align}
where $\alpha=\sum_{i=0}^{r-1}\{\langle\frac{-4ap^i}{q-1}\rangle+2\langle\frac{ap^i}{q-1}\rangle-\langle\frac{-2ap^i}{q-1}\rangle\}$.
Using Lemma \ref{lemma3_1} for $t=4$ and $t=2$, we deduce that
\begin{align}
A&=\frac{1}{q^2}\sum_{a=0}^{q-2}\omega^a(x)\pi^{(p-1)\alpha}\prod_{i=0}^{r-1}\frac{\Gamma_p(\langle(\frac{1}{4}-\frac{a}{q-1})p^i\rangle)
\Gamma_p(\langle(\frac{3}{4}-\frac{a}{q-1})p^i\rangle)\Gamma_p^2(\langle\frac{ap^i}{q-1}\rangle)}{\Gamma_p(\langle\frac{p^i}{4}\rangle)
\Gamma_p(\langle\frac{3p^i}{4}\rangle)}\notag\\
&\hspace{1cm}-\frac{q-1}{q^2}\varphi(-x).\notag
\end{align}
Finally, using Lemma \ref{lemma3.2} we have
\begin{align}
A=-\frac{q-1}{q^2}\cdot{_2G}_{2}\left[\begin{array}{cc}
           \frac{1}{4}, & \frac{3}{4} \\
           0, & 0
         \end{array}|\frac{1}{x}\right]_q-\frac{q-1}{q^2}\varphi(-x).\notag
\end{align}
This completes the proof of the proposition.
\end{proof}
\begin{proposition}\label{prop-2}
Let $p$ be an odd prime and $x\in \mathbb{F}_q$. Then, for $x\neq 1$, we have
\begin{align}
\sum_{y\in\mathbb{F}_q}\varphi(y)\varphi(1-2y+xy^2)&=2\varphi(x-1)+\frac{q^2}{q-1}\sum_{\chi\in\widehat{\mathbb{F}_q^{\times}}}
{\varphi\chi^2\choose \chi}{\varphi\chi\choose \chi^2}\chi(x-1)\notag\\
&=-{_2G}_{2}\left[\begin{array}{cc}
           \frac{1}{4}, & \frac{3}{4} \\
           0, & 0
         \end{array}|\frac{1}{1-x}\right]_q.\notag
\end{align}
\end{proposition}
\begin{proof} From \eqref{eq-0} and then using Lemma \ref{lemma2_2}, we have
\begin{align}\label{eq-5}
{\varphi\chi^2\choose \chi}{\varphi\chi\choose \chi^2}&=\frac{\chi(-1)}{q^2}J(\varphi\chi^2, \overline{\chi})J(\varphi\chi, \overline{\chi}^2)\notag\\
&=\frac{\chi(-1)}{q^2}\left[\frac{g(\varphi\chi^2)g(\overline{\chi})}{g(\varphi\chi)}+(q-1)\chi(-1)\delta(\varphi\chi)\right]\notag\\
&\hspace{1 cm}\times\left[\frac{g(\varphi\chi)g(\overline{\chi}^2)}{g(\varphi\overline{\chi})}+(q-1)\delta(\varphi\overline{\chi})\right].
\end{align}
From Lemma \ref{lemma2_1}, we have $g(\varphi)^2=q\varphi(-1)$. Since $\delta(\chi)=0$ for $\chi\neq \varepsilon, \delta(\varepsilon)=1$ and $g(\varepsilon)=-1$, \eqref{eq-5} yields
\begin{align}\label{eq-3.9}
B:=\sum_{\chi\in\widehat{\mathbb{F}_q^{\times}}}&
{\varphi\chi^2\choose \chi}{\varphi\chi\choose \chi^2}\chi(x-1)\notag\\&=\frac{1}{q^2}\sum_{\chi\in\widehat{\mathbb{F}_q^{\times}}}\frac{g(\varphi\chi^2)g(\overline{\chi})g(\overline{\chi}^2)}
{g(\varphi\overline{\chi})}\chi(1-x)-2\frac{q-1}{q^2}\varphi(x-1).
\end{align}
Using Lemma \ref{lemma2_2} and then \eqref{eq-0} we obtain
\begin{align}\label{eq-3.3}
\frac{g(\varphi\chi^2)g(\overline{\chi}^2)}{g(\varphi)}=q{\varphi\chi^2\choose \chi^2},
\end{align}
and
\begin{align}\label{eq-3.4}
\frac{g(\varphi)g(\overline{\chi})}
{g(\varphi\overline{\chi})}=q\chi(-1){\varphi\choose \chi}-(q-1)\chi(-1)\delta(\varphi\overline{\chi}).
\end{align}
From \eqref{eq-2}, we have ${\varphi\choose \varepsilon}=-\frac{1}{q}$. Hence, \eqref{eq-3.3} and \eqref{eq-3.4} yield
\begin{align}\label{eq-3.5}
&\frac{1}{q^2}\sum_{\chi\in\widehat{\mathbb{F}_q^{\times}}}\frac{g(\varphi\chi^2)g(\overline{\chi})g(\overline{\chi}^2)}
{g(\varphi\overline{\chi})}\chi(1-x)\notag\\
&=\sum_{\chi\in\widehat{\mathbb{F}_q^{\times}}}
{\varphi\chi^2\choose \chi^2}{\varphi\choose \chi}\chi(x-1)-\frac{q-1}{q}\sum_{\chi\in\widehat{\mathbb{F}_q^{\times}}}\chi(x-1){\varphi\chi^2\choose \chi^2}\delta(\varphi\overline{\chi})\notag\\
&=\sum_{\chi\in\widehat{\mathbb{F}_q^{\times}}}{\varphi\chi^2\choose \chi^2}{\varphi\choose \chi}\chi(x-1)-\frac{q-1}{q}{\varphi\choose \varepsilon}\varphi(x-1)\notag\\
&=\sum_{\chi\in\widehat{\mathbb{F}_q^{\times}}}{\varphi\chi^2\choose \chi^2}{\varphi\choose \chi}\chi(x-1)+\frac{q-1}{q^2}\varphi(x-1).
\end{align}
Applying \eqref{eq-0} on the right hand side of \eqref{eq-3.5}, and then \eqref{lemma2_4} we have
\begin{align}
&\frac{1}{q^2}\sum_{\chi\in\widehat{\mathbb{F}_q^{\times}}}\frac{g(\varphi\chi^2)g(\overline{\chi})g(\overline{\chi}^2)}
{g(\varphi\overline{\chi})}\chi(1-x)\notag\\
&= \frac{1}{q}\sum_{\substack{\chi\in\widehat{\mathbb{F}_q^{\times}}\\ y\in\mathbb{F}_q}}{\varphi\choose \chi}\chi(x-1)\varphi\chi^2(y)\overline{\chi}^2(1-y)+\frac{q-1}{q^2}\varphi(x-1)\notag\\
&=\frac{1}{q}\sum_{\substack{\chi\in\widehat{\mathbb{F}_q^{\times}}\\ y\in\mathbb{F}_q, y\neq 1}}\varphi(y){\varphi\choose \chi}\chi\left(\frac{(x-1)y^2}{(1-y)^2}\right)+\frac{q-1}{q^2}\varphi(x-1)\notag\\
&=\frac{q-1}{q^2}\sum_{y\in\mathbb{F}_q, y\neq 1}\varphi(y)\left[\varphi\left(1+\frac{(x-1)y^2}{(1-y)^2}\right)-\delta\left(\frac{(x-1)y^2}{(1-y)^2}\right)\right]
+\frac{q-1}{q^2}\varphi(x-1)\notag\\
&=\frac{q-1}{q^2}\sum_{\substack{y\in\mathbb{F}_q\\y\neq1}}\varphi(y)\varphi(1-2y+xy^2)+\frac{q-1}{q^2}\varphi(x-1).\notag
\end{align}
Adding and subtracting the term under summation for $y=1$, we have
\begin{align}\label{eq-3.7}
&\frac{1}{q^2}\sum_{\chi\in\widehat{\mathbb{F}_q^{\times}}}\frac{g(\varphi\chi^2)g(\overline{\chi})g(\overline{\chi}^2)}
{g(\varphi\overline{\chi})}\chi(1-x)\notag\\
&=\frac{q-1}{q^2}\sum_{y\in\mathbb{F}_q}\varphi(y)\varphi(1-2y+xy^2).
\end{align}
Combining \eqref{eq-3.9} and \eqref{eq-3.7} we readily obtain the first equality of the proposition.
\par To complete the proof of the proposition, we relate the character sums given in \eqref{eq-3.9} to the $p$-adic hypergeometric series. 
Using Davenport-Hasse relation for $m=2, \psi=\chi^2$ and $m=2, \psi=\overline{\chi}$, we have
\begin{align}
g(\varphi\chi^2)=\frac{g(\chi^4)g(\varphi)\overline{\chi}^2(4)}{g(\chi^2)}\notag
\end{align}
and
\begin{align}
g(\varphi\overline{\chi})=\frac{g(\overline{\chi}^2)g(\varphi)\chi(4)}{g(\overline{\chi})},\notag
\end{align}
respectively. Plugging these two expressions in \eqref{eq-3.9} we obtain
\begin{align}
B&=\frac{1}{q^2}\sum_{\chi\in\widehat{\mathbb{F}_q^{\times}}}\frac{g(\chi^4)g^2(\overline{\chi})}{g(\chi^2)}\overline{\chi}^3(4)\chi(1-x)
-2\frac{(q-1)}{q^2}\varphi(x-1).\notag
\end{align}
Now, considering $\chi=\omega^a$ and then applying Gross-Koblitz formula we obtain
\begin{align}
B&=\frac{1}{q^2}\sum_{a=0}^{q-2}\omega^{a}(1-x)~\overline{\omega}^{3a}(4)\pi^{(p-1)\alpha}\prod_{i=0}^{r-1}
\frac{\Gamma_p(\langle\frac{-4ap^i}{q-1}\rangle)\Gamma_p^2(\langle\frac{ap^i}{q-1}\rangle)}{\Gamma_p(\langle\frac{-2ap^i}{q-1}\rangle)}
-2\frac{(q-1)}{q^2}\varphi(x-1).\notag
\end{align}
where $\alpha=\sum_{i=0}^{r-1}\{\langle\frac{-4ap^i}{q-1}\rangle+2\langle\frac{ap^i}{q-1}\rangle-\langle\frac{-2ap^i}{q-1}\rangle\}$.
Proceeding similarly as shown in the proof of Proposition \ref{prop-1}, we deduce that
\begin{align}
B=-\frac{q-1}{q^2}\cdot{_2G}_{2}\left[\begin{array}{cc}
           \frac{1}{4}, & \frac{3}{4} \\
           0, & 0
         \end{array}|\frac{1}{1-x}\right]_q-2\frac{q-1}{q^2}\varphi(x-1).\notag
\end{align}
This completes the proof of the proposition.
\end{proof}

\par
Before we prove our main results, we now recall the following definition of a finite field hypergeometric function introduced by McCarthy in \cite{mccarthy3}.
\begin{definition}\cite[Definition 1.4]{mccarthy3}
Let $A_0, A_1, \ldots A_n, B_1, B_2, \ldots, B_n\in\widehat{\mathbb{F}_q^{\times}}$. Then the ${_{n+1}F}_n(\cdots)^{\ast}$ 
finite field hypergeometric function over $\mathbb{F}_q$ is defined by
\begin{align}
&{_{n+1}F}_n\left(\begin{array}{cccc}
                A_0, & A_1, & \ldots, & A_n\\
                 & B_1, & \ldots, & B_n
              \end{array}\mid x \right)^{\ast}_q\notag\\
              &=\frac{1}{q-1}\sum_{\chi\in \widehat{\mathbb{F}_q^\times}}\prod_{i=0}^{n}\frac{g(A_i\chi)}{g(A_i)}\prod_{j=1}^n\frac{g(\overline{B_j\chi})}{g(\overline{B_j})}g(\overline{\chi})
\chi(-1)^{n+1}\chi(x).\notag
\end{align}
\end{definition}
The following proposition gives a relation between McCarthy's and Greene's finite field hypergeometric functions when certain conditions on the parameters are satisfied.
\begin{proposition}\cite[Proposition 2.5]{mccarthy3}\label{prop-300}
If $A_0\neq\varepsilon$ and $A_i\neq B_i$ for $1\leq i\leq n$, then
\begin{align}
&{_{n+1}F}_n\left(\begin{array}{cccc}
                A_0, & A_1, & \ldots, & A_n\\
                 & B_1, & \ldots, & B_n
              \end{array}\mid x \right)_q^{\ast}\notag\\
              &=\left[\prod_{i=1}^n{A_i \choose B_i}^{-1}\right]{_{n+1}F}_n\left(\begin{array}{cccc}
                A_0, & A_1, & \ldots, & A_n\\
                 & B_1, & \ldots, & B_n
              \end{array}\mid x \right)_q.\notag
\end{align}              
\end{proposition}
In \cite[Lemma 3.3]{mccarthy2}, McCarthy proved a relation between ${_{n+1}F}_n(\cdots)^{\ast}$ and the $p$-adic hypergeometric series 
${_{n}G}_n[\cdots]$. We note that the relation is true for $\mathbb{F}_q$ though it was proved for $\mathbb{F}_p$ in \cite{mccarthy2}. 
Hence, we obtain a relation between ${_{n}G}_n[\cdots]$ and the Greene's finite field hypergeometric functions due to Proposition \ref{prop-300}. In the following proposition, we 
list three such identities which will be used to prove our main results.
\begin{proposition}\label{prop-301}
Let $x\neq 0$. Then
\begin{align}
\label{tr1}{_2G}_{2}\left[\begin{array}{cc}
           \frac{1}{4}, & \frac{3}{4} \\
           0, & 0
         \end{array}|x\right]_q & = -q\cdot {_2}F_1\left(
         \begin{array}{cc}
           \chi_4, & \chi_4^3 \\
           ~ & \varepsilon \\
         \end{array}|\frac{1}{x}
       \right)_q;\\
\label{tr2}  {_2G}_{2}\left[\begin{array}{cc}
           \frac{1}{2}, & \frac{1}{2} \\
           0, & 0
         \end{array}|x\right]_q & = -q\cdot {_2}F_1\left(
         \begin{array}{cc}
           \varphi, & \varphi \\
           ~ & \varepsilon \\
         \end{array}|\frac{1}{x}
       \right)_q;\\
\label{tr3}       {_3G}_{3}\left[\begin{array}{ccc}
           \frac{1}{2}, & \frac{1}{2} &\frac{1}{2}\\
           0, & 0, &0
         \end{array}|x\right]_q & = q^2\cdot {_3}F_2\left(
         \begin{array}{ccc}
           \varphi, & \varphi, &\varphi \\
           ~ & \varepsilon, &\varepsilon \\
         \end{array}|\frac{1}{x}
       \right)_q.
\end{align}
We note that \eqref{tr1} is valid when $q\equiv 1\pmod{4}$. 
\end{proposition}
\begin{proof} 
Applying \cite[Lemma 3.3]{mccarthy2} we have 
\begin{align}\label{tr_eq3}
{_2}F_1\left(
         \begin{array}{cc}
           \chi_4, & \chi_4^3 \\
           ~ & \varepsilon \\
         \end{array}|\frac{1}{x}
       \right)_q^{\ast}={_2G}_{2}\left[\begin{array}{cc}
           \frac{1}{4}, & \frac{3}{4} \\
           0, & 0
         \end{array}|x\right]_q.
\end{align}
From \eqref{eq-2}, we have ${\chi_4^3 \choose \varepsilon}=\frac{-1}{q}$. Using this value and Proposition \ref{prop-300} we find that
\begin{align}\label{tr_eq1}
{_2}F_1\left(
         \begin{array}{cc}
           \chi_4, & \chi_4^3 \\
           ~ & \varepsilon \\
         \end{array}|\frac{1}{x}
       \right)_q=-\frac{1}{q}{_2}F_1\left(
         \begin{array}{cc}
           \chi_4, & \chi_4^3 \\
           ~ & \varepsilon \\
         \end{array}|\frac{1}{x}
       \right)_q^{\ast}.
\end{align}
Now, combining \eqref{tr_eq3} and \eqref{tr_eq1} we readily obtain \eqref{tr1}. 
Proceeding similarly we deduce \eqref{tr2} and \eqref{tr3}. This completes the proof. 
\end{proof}
We now prove our main results.
\begin{proof}[Proof of Theorem \ref{MT-1}]
From Proposition \ref{prop-1} and Proposition \ref{prop-2} we have
\begin{align}
\sum_{y\in\mathbb{F}_q}\varphi(y)\varphi(1-2y+xy^2)&=-\varphi(-2)\cdot{_2G}_{2}\left[\begin{array}{cc}
           \frac{1}{4}, & \frac{3}{4} \\
           0, & 0
         \end{array}|\frac{1}{x}\right]_q\notag\\
         &=-{_2G}_{2}\left[\begin{array}{cc}
           \frac{1}{4}, & \frac{3}{4} \\
           0, & 0
         \end{array}|\frac{1}{1-x}\right]_q, \notag
\end{align}
which readily gives the desired transformation.
\end{proof}
\begin{proof}[Proof of Theorem \ref{MT-4}]
From \cite[Eq. 4.5]{GS} we have
\begin{align}\label{eq-7.1}
\varphi((1-u)/u){_3}F_{2}\left(
                           \begin{array}{ccc}
                             \varphi, & \varphi, & \varphi \\
                              & \varepsilon, & \varepsilon \\
                           \end{array}|\frac{u}{u-1}
                         \right)_p=&\varphi(u)f(u)^2+2\frac{\varphi(-1)}{p}f(u)-\frac{p-1}{p^2}\varphi(u)\notag\\
&+\frac{p-1}{p^2}\delta(1-u),
\end{align}
where $u=\frac{x}{x-1}, x\neq 1$ and
\begin{align}
f(u):=\frac{p}{p-1}\sum_{\chi\in\widehat{\mathbb{F}_p^{\times}}}
{\varphi\chi^2\choose\chi}{\varphi\chi\choose\chi}\chi\left(\frac{u}{4}\right).\notag
\end{align}
From \eqref{tr3} and \eqref{eq-7.1}, we have
\begin{align}\label{eq-7.6}
\frac{\varphi((1-u)/u)}{p^2}\cdot{_3}G_{3}\left[
                           \begin{array}{ccc}
                             \frac{1}{2}, & \frac{1}{2}, & \frac{1}{2} \\
                            0, & 0, & 0 \\
                           \end{array}|\frac{u-1}{u}
                         \right]_p& = \varphi(u)f(u)^2+2\frac{\varphi(-1)}{p}f(u)-\frac{p-1}{p^2}\varphi(u)\notag\\
&+\frac{p-1}{p^2}\delta(1-u).
\end{align}
Now, Proposition \ref{prop-1} gives
\begin{align}\label{eq-7.2}
f(u)=\frac{-\varphi(-u)}{p}-\frac{1}{p}\cdot{_2}G_2\left[\begin{array}{cc}
           \frac{1}{4}, & \frac{3}{4} \\
           0, & 0
         \end{array}|\frac{1}{u}\right]_p.
\end{align}
Finally, combining \eqref{eq-7.6} and \eqref{eq-7.2} and then putting $u=\frac{x}{x-1}$ we obtain the desired result. This completes the proof of the theorem.
\end{proof}
\begin{proof}[Proof of Theorem \ref{MT-3}]
Let $A=B=\varphi$ and $x\neq0,\pm1$. Then \cite[Thm. 4.16]{greene} yields
\begin{align}\label{eq-6.1}
{_2}F_1\left(
         \begin{array}{cc}
           \varphi, & \varphi \\
           ~ & \varepsilon \\
         \end{array}|x
       \right)_q=&\frac{\varphi(-1)}{q}\varphi(x(1+x))\notag\\
&+\varphi(1+x)\frac{q}{q-1}\sum_{\chi\in\widehat{\mathbb{F}_q^{\times}}}
{\varphi\chi^2\choose\chi}{\varphi\chi\choose\chi}\chi\left(\frac{x}{(1+x)^2}\right).
\end{align}
Now, using Proposition \ref{prop-1} we have
\begin{align}\label{eq-6.2}
 \sum_{\chi\in\widehat{\mathbb{F}_q^{\times}}}
{\varphi\chi^2\choose\chi}{\varphi\chi\choose\chi}\chi\left(\frac{x}{(1+x)^2}\right)=&-\frac{q-1}{q^2}\varphi\left(\frac{-4x}{(1+x)^2}\right)\notag\\
&-\frac{q-1}{q^2}\cdot
{_2}G_2\left[\begin{array}{cc}
           \frac{1}{4}, & \frac{3}{4} \\
           0, & 0
         \end{array}|\frac{(1+x)^2}{4x}\right]_q.
\end{align}
Applying Theorem \ref{MT-1} on the right hand side of \eqref{eq-6.2} we obtain
\begin{align}\label{eq-6.3}
 \sum_{\chi\in\widehat{\mathbb{F}_q^{\times}}}
{\varphi\chi^2\choose\chi}{\varphi\chi\choose\chi}\chi\left(\frac{x}{(1+x)^2}\right)=&-\frac{q-1}{q^2}\varphi\left(-x\right)\notag\\
&-\frac{q-1}{q^2}\varphi(-2)\cdot
{_2}G_2\left[\begin{array}{cc}
           \frac{1}{4}, & \frac{3}{4} \\
           0, & 0
         \end{array}|\frac{(1+x)^2}{(1-x)^2}\right]_q.
\end{align}
Combining \eqref{eq-6.1} and \eqref{eq-6.3} we have
\begin{align}\label{eq-6.4}
{_2}G_2\left[\begin{array}{cc}
           \frac{1}{4}, & \frac{3}{4} \\
           0, & 0
         \end{array}|\frac{(1+x)^2}{(1-x)^2}\right]_q=-q\varphi(-2)\varphi(1+x)\cdot{_2}F_1\left(
         \begin{array}{cc}
           \varphi, & \varphi \\
           ~ & \varepsilon \\
         \end{array}|x
       \right)_q,
\end{align}
which completes the proof of the theorem due to \eqref{tr2}.
\end{proof}
We finally present the proof of Theorem \ref{MT-2}.
\begin{proof}[Proof of Theorem \ref{MT-2}]
Let $q\equiv1\pmod{4}$. Then we readily obtain the desired transformation for the finite field hypergeometric functions
from \eqref{eqn-MT-3} using \eqref{tr1} and \eqref{tr2}.
\end{proof}
\section{Special values of ${_2}G_2[\cdots]$}
Finding special values of hypergeometric function is an important and interesting problem. Only a few special values of
the ${_n}G_n$-functions are known (see for example \cite{BSM}). In \cite{BSM}, the authors with McCarthy obtained some special values of ${_n}G_n[\cdots]$ when $n=2, 3, 4$.
From \eqref{eq-6.4}, for any odd prime $p$ and $x\neq 0, \pm 1$, we have
\begin{align}\label{eq-6.4-1}
{_2}G_2\left[\begin{array}{cc}
           \frac{1}{4}, & \frac{3}{4} \\
           0, & 0
         \end{array}|\frac{(1+x)^2}{(1-x)^2}\right]_q=-q\varphi(-2)\varphi(1+x)\cdot{_2}F_1\left(
         \begin{array}{cc}
           \varphi, & \varphi \\
           ~ & \varepsilon \\
         \end{array}|x
       \right)_q.
\end{align}
Values of the finite field hypergeometric function ${_2}F_1\left(
         \begin{array}{cc}
           \varphi, & \varphi \\
           ~ & \varepsilon \\
         \end{array}|x
       \right)_q$ are obtained for many values of $x$. For example, see Barman and Kalita \cite{BK1, BK3}, Evans and Greene \cite{evans2}, Greene \cite{greene},
       Kalita \cite{GK}, and Ono \cite{ono}.
\begin{proof}[Proof of Theorem \ref{MT-5}]
 Let $\lambda \in \{-1, \frac{1}{2}, 2\}$. If $p$ is an odd prime, then from \cite[Thm. 2]{ono} we have
        \begin{align*}
{_2}F_1\left(
         \begin{array}{cc}
           \varphi, & \varphi \\
           ~ & \varepsilon \\
         \end{array}|\lambda
       \right)_p=\left\{
                                             \begin{array}{ll}
                                               0, & \hbox{if $p\equiv 3\pmod{4}$;} \\
                                               \frac{2x(-1)^{\frac{x+y+1}{2}}}{p}, & \hbox{if $p\equiv 1\pmod{4}$, $x^2+y^2=p$, and $x$ odd.}
                                             \end{array}
                                           \right.
\end{align*}
Putting the above values for $\lambda=\frac{1}{2}, 2$ into \eqref{eq-6.4-1} we readily obtain the required values of the ${_n}G_n$-function.
\par Let $q\equiv 1\pmod{4}$. Then from \eqref{tr1} we have
\begin{align*}
{_2}F_1\left(
         \begin{array}{cc}
           \chi_4, & \chi_4^3 \\
           ~ & \varepsilon \\
         \end{array}|\frac{1}{9}
       \right)_q=-\frac{1}{q}~{_2}G_2\left[\begin{array}{cc}
           \frac{1}{4}, & \frac{3}{4} \\
           0, & 0
         \end{array}|9\right]_q.
       \end{align*}
From the above identity we readily obtain the required value of the finite field hypergeometric function.
This completes the proof of the theorem.
\end{proof}
We now have the following corollary.
\begin{cor}\label{cor-1}
Let $p\equiv 1\pmod{4}$. We have
\begin{align*}
       {\chi_4\choose \varphi}+ {\chi_4^3\choose \varphi}=
                                               \frac{2x(-1)^{\frac{x+y+1}{2}}}{p},
       \end{align*}
       where  $x^2+y^2=p$ and $x$ is odd.
\end{cor}
\begin{proof}
From Theorem \ref{MT-5} and \cite[Thm. 1.4 (i)]{BK1} we have
\begin{align*}
       {\chi_4\choose \varphi}+ {\chi_4^3\choose \varphi}=
                                               \frac{2x\varphi(2)\chi_4(-1)(-1)^{\frac{x+y+1}{2}}}{p},
       \end{align*}
       where  $x^2+y^2=p$ and $x$ is odd.
       Let $m$ be the order of $\chi\in \widehat{\mathbb{F}_q^{\times}}$.
       We know that $\chi(-1)=-1$ if and only if $m$ is even and $(q-1)/m$ is odd.
       Since $p\equiv 1\pmod{4}$, therefore, either $p\equiv 1\pmod{8}$ or $p\equiv 5\pmod{8}$. If
       $p\equiv 1\pmod{8}$, then $\varphi(2)=\chi_4(-1)=1$. Also, if $p\equiv 5\pmod{8}$, then $\varphi(2)=\chi_4(-1)=-1$. Hence, in both the cases
       $\varphi(2)\cdot\chi_4(-1)=1$. This completes the proof.
\end{proof}

\begin{proof}[Proof of Theorem \ref{MT-7}]
From \cite[Thm. 1.1]{GK}, for $q\equiv 1\pmod{8}$, we have
\begin{align}\label{eq-8.2}
{_2}F_1\left(
         \begin{array}{cc}
           \varphi, & \varphi \\
           ~ & \varepsilon \\
         \end{array}|\frac{4\sqrt{2}}{2\sqrt{2}\pm3}\right)_q=\varphi(3\pm2\sqrt{2})\left\{{\chi_4\choose \varphi}+{\chi_4^3\choose \varphi}\right\}.
       \end{align}
       Now, comparing \eqref{eq-6.4} and \eqref{eq-8.2} for $x=\frac{4\sqrt{2}}{2\sqrt{2}\pm3}$ we obtain \eqref{eq-8.5}. Similarly, using
       \cite[Thm. 1.1]{GK} and \eqref{eq-6.4} for $x=\frac{4}{2\pm\sqrt{3}}$ we derive \eqref{eq-8.6} and \eqref{eq-8.7}.
\end{proof}
\begin{proof}[Proof of Theorem \ref{MT-6}]
From \eqref{tr1}, we have
\begin{align}\label{eq-8.8}
       {_2}F_1\left(
         \begin{array}{cc}
           \chi_4, & \chi_4^3 \\
           ~ & \varepsilon \\
         \end{array}|\left(\frac{-2\sqrt{2}\pm3}{6\sqrt{2}\pm3}\right)^2
       \right)_q=-\frac{1}{q}\cdot{_2}G_2\left[
         \begin{array}{cc}
           \frac{1}{4}, & \frac{3}{4} \\
           0, & 0 \\
         \end{array}|\left(\frac{6\sqrt{2}\pm3}{-2\sqrt{2}\pm3}\right)^2
       \right]_q.
       \end{align}
Comparing \eqref{eq-8.5} and  \eqref{eq-8.8} we readily obtain \eqref{eq-8.3}.
Again, we have
\begin{align}\label{eq-8.9}
       {_2}F_1\left(
         \begin{array}{cc}
           \chi_4, & \chi_4^3 \\
           ~ & \varepsilon \\
         \end{array}|\left(\frac{-2\pm\sqrt{3}}{6\pm\sqrt{3}}\right)^2
       \right)_q=-\frac{1}{q}\cdot{_2}G_2\left[
         \begin{array}{cc}
           \frac{1}{4}, & \frac{3}{4} \\
           0, & 0 \\
         \end{array}|\left(\frac{6\pm\sqrt{3}}{-2\pm\sqrt{3}}\right)^2
       \right]_q.
       \end{align}
       Now, comparing \eqref{eq-8.7} and \eqref{eq-8.9} we deduce \eqref{eq-8.4}.
\end{proof}
Applying Corollary \ref{cor-1}, from \eqref{eq-8.5} and \eqref{eq-8.3} we have the following corollary.
\begin{cor}
Let $p\equiv 1\pmod{8}$. Then
\begin{align*}
       {_2}G_2\left[
         \begin{array}{cc}
           \frac{1}{4}, & \frac{3}{4} \\
           0, & 0 \\
         \end{array}|\left(\frac{6\sqrt{2}\pm3}{-2\sqrt{2}\pm3}\right)^2
       \right]_p &= -2x\varphi(6\pm 12\sqrt{2})(-1)^{\frac{x+y+1}{2}},
       \end{align*}
       where  $x^2+y^2=p$ and $x$ is odd.
\end{cor}


\begin{thebibliography}{99}

\bibitem{ahlgren}
 S. Ahlgren and K. Ono, {\it A  Gaussian  hypergeometric  series  evaluation  and  Ap\'{e}ry  number  congruences},
J. Reine Angew. Math. 518 (2000), 187--212.

\bibitem{bailey}
W. Bailey, {\it Generalized hypergeometric series}, Cambridge University Press, Cambridge, 1935.

\bibitem{BK} R. Barman and G. Kalita, {\it Hypergeometric functions over $\mathbb{F}_q$ and traces of Frobenius for elliptic curves}, Proc. Amer. Math. Soc. 141 (2013), no. 10,
3403--3410.

\bibitem{BK1}
R. Barman and G. Kalita, {\it Elliptic curves and special values of Gaussian hypergeometric series},
J. Number Theory 133 (2013), 3099--3111.

\bibitem{BK3} R. Barman and G. Kalita, {\it Certain values of Gaussian hypergeometric series and a family of algebraic curves},
Int. J. Number Theory 8 (2012), no. 4, 945--961.

\bibitem{BS2}
R. Barman and N. Saikia, {\it Certain Transformations for Hypergeometric series in the p-adic setting}, Int. J. Number Theory 11 (2015), no. 2, 645--660.

\bibitem{BS1}
R. Barman and N. Saikia, {\it $p$-Adic gamma function and the trace of Frobenius of elliptic curves},
J. Number Theory 140 (2014), no. 7, 181--195.

\bibitem{BSM}
R. Barman, N. Saikia and D. McCarthy, {\it Summation identities and special values of hypergeometric series in the $p$-adic setting}, J. Number Theory 153 (2015), 63--84.

\bibitem{evans}
B. Berndt, R. Evans, and K. Williams, {\it Gauss and Jacobi Sums}, Canadian Mathematical Society Series of Monographs and Advanced Texts,
A Wiley-Interscience Publication, John Wiley \& Sons, Inc., New York, (1998).

\bibitem{evans-mod}
R. Evans, {\it Hypergeometric $_3F_2(1/4)$ evaluations  over  finite  fields  and  Hecke  eigenforms}, Proc. Amer. Math. Soc. 138 (2010), no. 2, 517--531.

\bibitem{evans2} R. Evans and J. Greene, {\it Evaluation of Hypergeometric Functions over Finite Fields},
Hiroshima Math. J. 39 (2009), no. 2, 217-235.

\bibitem {EG} R. Evans and J. Greene, \textit{Clausen's theorem and hypergeometric functions over finite fields},
Finite Fields and Their Applications 15 (2009), 97--109.


\bibitem{frechette}
S. Frechette, K. Ono, and M. Papanikolas, {\it Gaussian  hypergeometric  functions  and  traces  of  Hecke operators}, Int. Math. Res. Not. 2004, no. 60, 3233--3262.

\bibitem{fuselier}
J.  Fuselier, {\it Hypergeometric  functions  over $\mathbb{F}_p$ and  relations  to  elliptic  curves  and  modular  forms}, Proc. Amer. Math. Soc. 138 (2010), no.1, 109--123.

\bibitem{Fuselier-McCarthy} J. Fuselier and D. McCarthy, {\it Hypergeometric type identities in the $p$-adic setting and modular forms},
Proc. Amer. Math. Soc. 144 (2016), 1493--1508


\bibitem{greene}
J. Greene, {\it Hypergeometric functions over finite fields}, Trans. Amer. Math. Soc. 301 (1987), no. 1, 77--101.

\bibitem{greene2}
J. Greene, {\it Character Sum Analogues for Hypergeometric and Generalized Hypergeometric Functions over Finite Fields},
Ph.D. thesis, Univ. of Minnesota, Minneapolis, 1984.

\bibitem{GS}
J. Greene and D. Stanton, {\it A Character Sum Evaluation and Gaussian Hypergeometric Series}, J. Number Theory 23 (1986), 136--148.

\bibitem{gross}
B. H. Gross and N. Koblitz, {\it Gauss sum and the $p$-adic $\Gamma$-function}, Annals of Mathematics 109 (1979), 569--581.

\bibitem{koike}
M. Koike, {\it Hypergeometric series over finite fields and Ap\'{e}ry numbers}, Hiroshima Math. J. 22 (1992), no. 3, 461--467.

\bibitem{GK}
G. Kalita, {\it Values of Gaussian hypergeometric functions and their connection to algebraic curves}, Int. J. Number Theory 14 (2018), no. 1, 1--18.


\bibitem{kob} N. Koblitz, {\it $p$-adic analysis: a short course on recent work}, London Math. Soc. Lecture
Note Series, 46. Cambridge University Press, Cambridge-New York, (1980).

\bibitem{lennon} C. Lennon, {\it Trace formulas for Hecke operators, Gaussian hypergeometric functions, and the modularity of a threefold}, 
J. Number Theory 131 (2011), no. 12, 2320--2351.

\bibitem{lennon2} C. Lennon, {\it Gaussian hypergeometric evaluations of traces of Frobenius for elliptic curves}, Proc. Amer. Math. Soc. 139 (2011), no. 6, 1931--1938.

\bibitem{mccarthy1}
D. McCarthy, {\it Extending Gaussian hypergeometric series to the $p$-adic setting}, Int. J. Number Theory 8 (2012), no. 7, 1581--1612.

\bibitem{mccarthy2}
D. McCarthy, {\it The trace of Frobenius of elliptic curves and the $p$-adic gamma function}, Pacific J. Math. 261 (2013), no. 1, 219--236.

\bibitem{mccarthy3}
D. McCarthy, {\it Transformations of well-poised hypergeometric functions over finite  fields}, Finite Fields and Their Applications, 18 (2012), no. 6, 1133--1147.

\bibitem{mccarthy4} D. McCarthy, {\it On a supercongruence conjecture of Rodriguez-Villegas}, Proc. Amer. Math. Soc. 140 (2012), 2241--2254.

\bibitem{mortenson}
E. Mortenson, {\it Supercongruences for truncated ${_{n+1}}F_n$-hypergeometric series with applications to certain weight three newforms}, Proc. Amer. Math. Soc.
133 (2005), no. 2, 321--330.

\bibitem{mc-papanikolas}
D. McCarthy and M. Papanikolas, {\it A finite field hypergeometric function associated to eigenvalues of a Siegel eigenform}, Int. J. Number Theory 11 (2015), no. 8, 2431--2450.

\bibitem{ono} K. Ono, \textit{Values of Gaussian hypergeometric series}, Trans. Amer. Math. Soc. 350 (1998), no. 3, 1205--1223.

\bibitem{salerno} A. Salerno, {\it Counting  points  over finite fields  and  hypergeometric  functions}, Funct. Approx. Comment. Math. 49 (2013), no. 1, 137--157.

\bibitem{vega} M. V. Vega, {\it Hypergeometric functions over finite fields and their relations to algebraic curves}, Int. J. Number Theory 7 (2011), no. 8, 2171--2195.
\end{thebibliography}
\end{document}